\documentclass[12pt]{amsart}
\usepackage{amsthm}
\usepackage{amsmath}
\usepackage{amssymb}
\usepackage{comment}
\usepackage{enumitem}
\usepackage{amsfonts}
\usepackage{mathtools}
\usepackage{setspace}
\usepackage[dvipsnames]{xcolor}
\usepackage[hidelinks]{hyperref}
\usepackage{dynkin-diagrams}
\hypersetup{
    colorlinks=true,
    linkcolor=blue,
    urlcolor=red,
    citecolor=Green
    }
\usepackage{algorithm} 
\usepackage{algorithmic}

\usepackage{subfigure}

\usepackage[utf8]{inputenc}

\advance\paperheight1.00cm
\advance\textheight1.00cm
\advance\vsize1.00cm
\advance\topskip-0.5cm
\advance\voffset-0.5cm
\usepackage{tikz,tikz-cd}
\usepackage{todonotes}
\DeclareTextFontCommand{\emph}{\color{blue}\em}

\newcommand{\precdot}{\mathrel{\ooalign{$\prec$\cr
  \hidewidth\raise0.001ex\hbox{$\cdot\mkern0.6mu$}\cr}}}

\newtheorem{theorem}{Theorem}[section]
\newtheorem{proposition}[theorem]{Proposition}

\newtheorem{lemma}[theorem]{Lemma}
\newtheorem{corollary}[theorem]{Corollary}

\theoremstyle{definition}
\newtheorem{ex}[theorem]{Example}
\newtheorem{definition}[theorem]{Definition}

\theoremstyle{remark}
\newtheorem*{remark}{Remark}

\DeclareMathOperator{\wt}{\mathrm{wt}}
\newcommand{\op}[1]{\overset{#1}{\rightsquigarrow}}
\DeclareMathOperator{\mul}{\mathrm{mul}}
\newcommand{\nb}{\mathcal{NB}}

\DeclareMathOperator{\supp}{S}
\newcommand{\di}[1]{B(#1)}
\newcommand{\none}{\mathbf{None}}

\title{Boolean Schubert Structure Coefficients}
\author{Yibo Gao}
\address{Beijing International Center for Mathematical Research, Peking University, Beijing 100871, China}
\email{gaoyibo@bicmr.pku.edu.cn}
\author{Hai Zhu}
\address{Department of Mathematics, UC San Diego, La Jolla, CA, 92093, USA}
\email{haz138@ucsd.edu}

\date{\today}

\begin{document}
\pagestyle{plain}
\begin{abstract}
The Schubert problem asks for combinatorial models to compute structure constants of the cohomology ring with respect to Schubert classes and has been an important open problem in algebraic geometry and combinatorics that guided fruitful research for decades. In this paper, we provide an explicit formula for the (equivariant) Schubert structure constants $c_{uv}^w$ across all Lie types when the elements $u,v,w$ are boolean. In particular, in type $A$, all Schubert structure constants on boolean elements are either $0$ or $1$.
\end{abstract}
\maketitle

\section{Introduction}\label{sec:intro}

Let $G$ be a complex, connected, reductive algebraic group and $B$ be a Borel subgroup of $G$ with a maximal torus $T$. The homogeneous space $G/B$ is called the \emph{generalized flag variety}, which admits a \emph{Bruhat decomposition} $\sqcup_{w\in W}X_w^{\circ}$ into open \emph{Schubert cells}, whose closures are the \emph{Schubert varieties} $\{X_w\:|\: w\in W\}$, indexed by the \emph{Weyl group} $W=N_G(T)/T$. Let $\sigma_w\in H^*(G/B;\mathbb{Z})$ be the Poincar\'e dual of the fundamental class of $X_w$. 

The \emph{Schubert problem} asks for combinatorial interpretations of the structure constants $c_{uv}^w\in\mathbb{Z}_{\ge 0}$ of $H^*(G/B;\mathbb{Z})$ appearing in the expansion $\sigma_u\cdot\sigma_v=\sum_w c_{uv}^w\sigma_w$. It has been a major open problem in algebraic geometry and combinatorics for decades, guiding fruitful research in recent years. We mention a few beautiful results here in the massive literature: the most classical Chevalley-Monk formula \cite{monk}, the Pieri rule \cite{sottile-pieri}, the separated descent case \cite{fan2025bumpless,huang-separated-descent,knutson-separated-descent-puzzles}, puzzle rules for the Grassmannian \cite{knutson-tao-puzzles,knutson-tao-woodward-puzzles}, a survey on the equivariant Schubert calculus of the Grassmannian \cite{robichaux-survey}, $2$ or $3$ step partial flag varieties \cite{buch-kresch-purbhoo-tamvakis-puzzle,buch-kresch-tamvakis-GW,knutson-zinn-justin-puzzle1}, and various others working in richer cohomology theories. 

The goal of this paper is to make progress towards the Schubert problem. We describe an explicit rule (Corollory~\ref{cor:main}) for the Schubert structure constants $c_{uv}^w$ across all Lie types when the element $w$ is \emph{boolean}, a previously unexplored family of the Schubert problem, with connections to the Pieri's rule \cite{sottile-pieri} and hook's rule \cite{postnikov-FK}. Interestingly, our formula illustrates certain ``multiplicity-freeness" in type $A_n$, where all $c_{uv}^w\in\{0,1\}$ when $w$ is boolean. Boolean elements play an important role in the study of Schubert calculus. The Schubert variety $X_w$ is a toric variety if and only if $w$ is boolean \cite{karupp-boolean}, and boolean elements can be used to study spherical Schubert varieties \cite{levi-spherical-1,levi-spherical-2}. 

We work in the generality of the torus equivariant cohomology ring $H_T^*(G/B;\mathbb{Z})$. Let $\{\xi_w\:|\: w\in W\}$ be the \emph{equivariant Schubert classes} and write $\xi_u\cdot\xi_v=\sum d_{uv}^w\xi_w$ where $d_{uv}^w\in \mathbb{Z}[\Lambda]=H_{T}^{*}(\mathrm{pt};\mathbb{Z})$ is the \emph{equivariant Schubert structure constant}. 

\begin{remark}
The Kostant-Kumar formula \cite[Theorem 4.15]{kostant1986nil} provides $d_{uv}^w$ with a recursive formula \cite[Theorem 4.2]{richmond2023nil}. Moreover, in the boolean case, the Kostant-Kumar formula reduces to a combinatorially positive formula as follows.

Let $\mathcal{A} \coloneqq \mathbb{Z}[\alpha_1,\cdots,\alpha_n]$ denote the polynomial algebra in the simple roots and define the divided difference operator $\partial_j \, : \, \mathcal{A} \rightarrow \mathcal{A}$ as
\[\partial_j(p) \coloneqq \frac{s_j(p) - p}{\alpha_j}.\]
Let $w \in W$ be a boolean element with support set $S(w)=\{ i_1,i_2,\cdots,i_k \} $ and fix a reduced word $w = s_{i_1}s_{i_2}\cdots s_{i_k}$. Consider $u,v \le w$ and note that $u,v$ must be boolean as well. Furthermore, $u,v$ are uniquely determined by their support sets $S(u),S(v) \subseteq S(w)$. According to \cite[Subsection 4.4]{richmond2023nil}, the Kostant-Kumar formula says that
\begin{equation}\label{eq:kk-formula}
    d_{u,v}^w = B_{i_1} \circ B_{i_2} \circ \cdots \circ B_{i_k}(1)
\end{equation}
where the operator $B_j \, : \, \mathcal{A} \rightarrow \mathcal{A}$ is defined as
\[B_j(p)\coloneqq\begin{cases}
    \alpha_j \cdot s_j (p) &\text{if $s_j \in S(u) \cap S(v)$}\\
    s_j(p) &\text{if $s_j \in S(u) \bigtriangleup S(v)$}\\
    \partial_j(p) &\text{if $s_j \not\in S(u) \cup S(v)$}.
\end{cases}\]
Now if $p \in \mathcal{A}$ with non-negative coefficients does not contain the variable $\alpha_j$, then it can be shown using the twisted Leibniz formula \cite[Theorem 11.1.7 part (h)]{kumar2002kac} that $\partial_j(p)$ is a polynomial with non-negative coefficients in the simple roots. Hence Equation~\eqref{eq:kk-formula} is a positive formula in the boolean case since the root variable $\alpha_j$ only gets introduced when applying $B_j$.

However, it is not immediate from Equation~\eqref{eq:kk-formula} that boolean structure constants $c_{uv}^w \in \{ 0,1\}$ in type A, which will be proved in Corollary~\ref{cor:0-1}. 
\end{remark}

The following is our main theorem.

\begin{theorem}\label{thm:main}
    For boolean elements $u,v,w\in W$, 
    \[
        d_{uv}^w =\begin{cases}
            \sum\limits_{u\op{\supp(v)}w }\mul(u\op{\supp(v)}w)\cdot\wt(u\op{\supp(v)}w), &\text{if there exists a boolean insertion path $v\op{\supp(u)}w$} \\ 
            
            0, &\text{otherwise}
        \end{cases} 
    \]
    where the summation is over all boolean insertion paths $u\op{\supp(v)}w$.
\end{theorem}
The \emph{boolean insertion path} consists of the \emph{boolean insertion steps} that encode the equivariant Chevalley rule on boolean elements. These steps are also called the \emph{$k$-Bruhat order} in $H^*(\mathrm{Fl}_n;\mathbb{Z})$, and have been very useful in the Schubert problem \cite{lenart-sottile-skew-Schubert,lenart-growth-diagram}. The appearance of our boolean insertion path in the formula is not surprising. Interestingly, for boolean elements, these paths precisely govern the structure constants in a subtraction-free and multiplicity-free way (Proposition~\ref{prop:insert}). 

The precise definitions in Theorem~\ref{thm:main} are given in Definitions~\ref{def:path-wt} and \ref{def:path-mul}. We remark that $\mul(u\op{\supp(v)}w)\wt(u\op{\supp(v)}w)$ can be replaced by $\mul(v\op{\supp(u)}w)\wt(v\op{\supp(u)}w)$ in Theorem~\ref{thm:main}, making the formula symmetric.

We also have an ordinary cohomology version of Theorem~\ref{thm:main}.
\begin{corollary}\label{cor:main}
     For boolean elements $u,v,w\in W$,
    \[c_{uv}^w \! = \! \begin{cases}
            \sum\limits_{u\op{\supp(v)}w }\!\mul(u\!\op{\supp(v)}\!w),\! &\text{if there exists a non-equivariant boolean insertion path $v\!\op{\supp(u)}\!w$}\\ 
            
            0, &\text{otherwise}
        \end{cases} \] where the summation is over all non-equivariant boolean insertion paths $u\op{\supp(v)}w$.
\end{corollary}
Furthermore, $c_{uv}^w\in\{0,1\}$ in type $A$ (Corollary~\ref{cor:0-1}). 

This paper is organized as follows. In Section~\ref{sec:prelim}, we provide the necessary background on root systems, Weyl groups, the equivariant Chevalley formula, boolean elements and their boolean diagrams. In Section~\ref{sec:proof}, we introduce the boolean insertion algorithms and prove the main theorem (Theorem~\ref{thm:main}). In Section~\ref{sec:algo}, we give a fast algorithm to compute $c_{uv}^w$ for boolean elements.

\section{Preliminaries}\label{sec:prelim}
\subsection{Root systems and Weyl groups}
Let $\Phi:=\Phi(\mathfrak{g},T)$ be the \emph{root system} of weights for the adjoint action of $T$ on the Lie algebra $\mathfrak{g}$ of $G$, with a decomposition $\Phi^+\sqcup\Phi^-$ into \emph{positive roots} and negative roots. Let $\Delta=\{\alpha_1,\ldots,\alpha_r\}\subseteq\Phi^+$ be the corresponding set of \emph{simple roots}, which is a basis of $\mathfrak{h}_{\mathbb{R}}^{*}$, the real span of all roots. Let $\langle\cdot,\cdot\rangle$ be the nondegenerate scalar product on $\mathfrak{h}_{\mathbb{R}}^{*}$ induced by the Killing form. For each root $\alpha\in\Phi$, denote by $s_{\alpha}$ the corresponding reflection. Explicitly, we have \[s_{\alpha}\gamma=\gamma-\frac{2\langle\alpha,\gamma\rangle}{\langle\alpha,\alpha\rangle}\alpha.\]
For simplicity of notation, write the \emph{simple reflections} as $s_i:=s_{\alpha_i}$ for $\alpha_i\in\Delta$. For each root $\alpha\in\Phi$, we have a \emph{coroot} $\alpha^{\vee}=2\alpha/\langle\alpha,\alpha\rangle$. The \emph{fundamental weights} $\{\omega_{\alpha}\:|\:\alpha\in\Delta\}$ are the dual basis to the simple coroots $\{\alpha^{\vee}\:|\:\alpha\in\Delta\}$. Let $\Lambda$ be the weight space and we identify $\mathbb{Z}[\Lambda]$ as $\mathbb{Z}[\mathbf{t}]$, the polynomial ring in $\{t_{\alpha}:=\omega_{\alpha}-s_{\alpha}(\omega_{\alpha})\:|\:\alpha\in\Delta\}$. Here $s_\alpha(\omega_\alpha)$ is given by the natural action of the Weyl group $W$ on $\Lambda$ defined by $\langle w(\omega_\alpha),\beta\rangle\coloneqq \langle\omega_\alpha ,w^{-1}(\beta)\rangle$ for $w\in W$ and $\alpha,\beta\in\Delta$.

The following definition of directed Dynkin diagrams may slightly differ from the classical definition of Dynkin diagrams.

\begin{definition}\label{def:dynkin}
The \emph{directed Dynkin diagram} of $\Phi$ with a choice of simple root $\Delta$ is a directed graph whose vertex set is $\Delta$ with $-2\langle\alpha,\beta\rangle/\langle\beta,\beta\rangle\in\mathbb{N}$ edges going from $\alpha$ to $\beta$ for $\alpha\neq\beta\in\Delta$.
\end{definition}

\begin{ex}\label{ex:dynkin}
Consider the directed Dynkin diagram of type $C_3$. There are TWO edges connecting $\alpha_1$ and $\alpha_2$, one going from $\alpha_1$ to $\alpha_2$ and the other going from $\alpha_2$ to $\alpha_1$, although the classical Dynkin diagram of type $C_3$ (Figure~\ref{fig:dynkin}) only shows one edge between $\alpha_1$ and $\alpha_2$. Similarly, there are THREE edges between $\alpha_2$ and $\alpha_3$, one going from $\alpha_2$ to $\alpha_3$ and the other two going from $\alpha_3$ to $\alpha_2$, although the classical Dynkin diagram of type $C_3$ (Figure~\ref{fig:dynkin}) only shows two edge between $\alpha_2$ and $\alpha_3$. Note that whenever we draw a classical Dynkin diagram, we in fact use Definition~\ref{def:dynkin} to understand the directed structure on it as mentioned above.
    \begin{figure}[h!]
    \centering
    \begin{dynkinDiagram}[label,label macro/.code={\alpha_{\drlap{#1}}},edge length=.75cm,mark=o]C3
    \end{dynkinDiagram}
    \caption{The directed Dynkin diagram of type $C_3$.}
    \label{fig:dynkin}
    \end{figure}
\end{ex}

The \emph{Weyl group} is generated by the reflections $\{s_{\beta}\:|\:\beta\in\Phi\}$. It is equipped with a \emph{Coxeter length} function $\ell(\cdot)$ where $\ell(w)=\min\{\ell\:|\: w=s_{i_1}\cdots s_{i_{\ell}}\}$. Such an expression $w=s_{i_1}\cdots s_{i_{\ell}}$ is called a \emph{reduced word} of $w$ if $\ell(w)=\ell$. The \emph{Bruhat order} on $W$ is generated by $w<ws_{\beta}$ if $\ell(w)<\ell(ws_{\beta})$ for $\beta\in\Phi^+$. For $w\in W$, its \emph{support} is \[\supp(w):=\{\alpha_i\:|\: s_i\text{ appears in some reduced word of }w\}.\]
\begin{remark}
    Simple reflection $s_i$ appears in some reduced word of $w$ if and only if $s_i$ appears in all reduced words of $w$. Consequently, we have an equivalent definition
    \[\supp(w):=\{\alpha_i\:|\: s_i\text{ appears in all reduced words of }w\}.\]
\end{remark} 

A straightforward calculation gives us the following.
\begin{proposition}\label{prop:game}
    For $\alpha\in\Delta$ and $\beta=\sum\limits_{\alpha_i\in\Delta}n_i\alpha_i$,
    $s_{\alpha}(\beta)=\sum_{\alpha_i\in\Delta}n_i^{\prime}\alpha_i$
    where $n_i^{\prime}=n_i$ for $\alpha_i\neq\alpha$ and $n_j^{\prime}=\sum\limits_{k\neq j}\left(-\frac{2\langle\alpha_k,\alpha_j\rangle}{\langle\alpha_j,\alpha_j\rangle}n_k\right)-n_j$ for $\alpha_j=\alpha$.
\end{proposition}
Intuitively, to obtain $s_{\alpha}\beta$ from $\beta$, both written in the basis of $\Delta$, we replace the coefficient of $\alpha$ by the sum of the coefficients of the neighbors of $\alpha$, weighted by the number of edges going from each neighbor of $\alpha$ to $\alpha$ in the directed Dynkin diagram, minus the coefficient of $\alpha$ itself in $\beta$.

The following result \cite[p.351, Theorem 19.1.2]{anderson2023equivariant} lets us do calculations in $H^*_T(G/B;\mathbb{Z})$.
\begin{theorem}[Equivariant Chevalley formula]\label{thm:chevalley}
For $\alpha\in\Delta$ and $v \in W$,
\[
\xi_v\cdot\xi_{s_{\alpha}}=(\omega_{\alpha}-v(\omega_{\alpha}))\xi_v+\sum_{\substack{w=vs_{\beta}\\\ell(w)=\ell(v)+1}}\langle\omega_{\alpha},\beta^{\vee}\rangle\ \xi_w
\]
in $H_T^{*}(G/B;\mathbb{Z})$, where we sum over positive roots $\beta\in\Phi^+$.
\end{theorem}

\subsection{Boolean elements}
\begin{definition}
A Weyl group element $w\in W$ is \emph{boolean} if its lower Bruhat interval $[\mathrm{id},w]$ is isomorphic to a boolean lattice.
\end{definition}
The following Lemma is straightforward by the subword property \cite[Theorem 2.2.2]{Bjorner-Brenti}. See also \cite[Proposition 7.3]{tenner-boolean} and \cite[Proposition 3,1]{gao2020boolean}.
\begin{lemma}
An element $w\in W$ is boolean if and only if $w$ is a product of distinct simple reflections. In other words, $w$ is boolean if and only if $\ell(w)=|\supp(w)|$.
\end{lemma}
We now view boolean elements visually using \emph{boolean diagrams}.
\begin{definition}\label{def:boolean-diagram}
For $w\in W$ that is boolean, its \emph{boolean diagram} $\di{w}$ is a directed graph on $\supp(w)$ such that $\alpha_{k}\rightarrow\alpha_{j}$ if $\alpha_j$ is connected to $\alpha_k$ by an edge in the Dynkin diagram of $W$ and $s_{j}$ appears before $s_{k}$ in any reduced words of $w$.
\end{definition}
\begin{remark}
    $\alpha_k\rightarrow\alpha_j$ in $B(w)$ indicates that $s_j$ appears before $s_k$ in all reduced words of $w$.
\end{remark}
For two boolean diagrams $B(u)$ and $B(w)$, we write $B(u)\subseteq B(w)$ if $\supp(u)\subseteq\supp(w)$ and if $\alpha_k\rightarrow\alpha_j$ in $B(u)$, we also have $\alpha_k\rightarrow\alpha_j$ in $B(w)$.

\begin{ex}\label{ex:boolean-diagram-example}
Consider $w=s_3s_2s_4s_5s_7$ in $W(E_7)$. The directed Dynkin diagram of type $E_7$ and the boolean diagram $B(w)$ marked with solid nodes are shown in Figure~\ref{fig:boolean-diagram-example}.
\begin{figure}[h!]
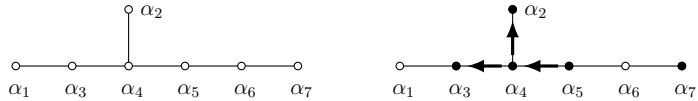

\centering
\dynkin[label,label macro/.code={\alpha_{\drlap{#1}}},edge
    length=.75cm,mark=o]E7
\qquad
\begin{dynkinDiagram}[label,label macro/.code={\alpha_{\drlap{#1}}},edge length=.75cm,mark=o]E7
        \foreach\r in {2,3,4,5,7} {\dynkinRootMark{*}\r}
        \draw[very thick, black,-latex]
            (root 4) to (root 3);
        \draw[very thick, black,-latex]
            (root 4) to (root 2);
        \draw[very thick, black,-latex]
            (root 5) to (root 4);
    \end{dynkinDiagram}
\caption{Left: the directed Dynkin diagram of type $E_7$. Right: the boolean diagram $B(w)$ for the boolean element $w=s_3s_2s_4s_5s_7$.}
\label{fig:boolean-diagram-example}
\end{figure}
\end{ex}
Let $\nb$ be the $\mathbb{Z}[\Lambda]$ linear subspace of $H_T^*(G/B;\mathbb{Z})$ spanned by the equivariant Schubert classes $\xi_w$ such that $w$ is not boolean. Note that the equivariant structure constant $d_{uv}^w$ is nonzero only when $u\leq w$, and that if $u\leq w$ and $u$ is not boolean, then $w$ is not boolean. Thus, $\nb$ is an ideal of $H_T^*(G/B;\mathbb{Z})$.

\section{Schubert structure constants for boolean elements}\label{sec:proof}
\subsection{The boolean insertion algorithms}
Now we define an operation that transforms one boolean element $u\in W$ into another boolean element $v\in W$ with respect to a simple root $\alpha$, which is denoted by $u\op{\alpha}v$. In fact, it encodes Equivariant Chevalley formula in Theorem~\ref{thm:chevalley} restricted to boolean elements. In particular, restricting to the cohomology ring in type $A$, $u\op{k}v$ exactly means that $v$ covers $u$ under the $k$-Bruhat order. We associate each operation with a multiplicity $\mul(u\op{\alpha}v)\in\mathbb{N}$ and weight $\wt(u\op{\alpha}v)\in\mathbb{Z}[\mathbf{t}]=\mathbb{Z}[\Lambda]$ as a nonzero polynomial with non-negative coefficients. Here the indeterminates
$t_{\gamma}=\omega_{\gamma}-s_{\gamma}(\omega_{\gamma})$ are indexed by simple roots.

All the paths in this article cannot pass a vertex repeatedly.

\begin{definition}\label{def:op}
For boolean elements $u,v\in W$ and $\alpha\in\Delta$, we write $u\op{\alpha}v$ and call it a \emph{boolean insertion} if one of the following mutually exclusive events happens:
\begin{enumerate}
\item\label{ev:1} $\alpha\in\supp(u)$, $\ell(v)=\ell(u)+1$, $\di{u}\subseteq\di{v}$ and there is a directed path in $\di{v}$ from $\alpha$ to the unique vertex of $\di{v}\setminus\di{u}$. In this case, $\wt(u\op{\alpha}v)\coloneqq1$.
\item\label{ev:2} $\alpha\in\supp(u)$ and $u=v$. In this case, $\wt(u\op{\alpha}v)\coloneqq\sum\limits_{L}t_{\gamma}$, summing over all directed paths $L$ of the directed Dynkin diagram from $\alpha$ to some vertex $\gamma\in\supp(u)$, which is compatible with the direction of $B(u)$. Here $L$ is permitted to has length $0$.
\item\label{ev:3} $\alpha\not\in\supp(u)$, $\ell(v)=\ell(u)+1$ and $B(u)\subseteq B(v)$ where $\alpha$ is the unique vertex of $B(v)\setminus B(u)$. In this case, $\wt(u\op{\alpha}v)\coloneqq1$.
\end{enumerate}
We say $u\op{\alpha}v$ is \emph{non-equivariant} if \eqref{ev:1} or \eqref{ev:3} happens and is \emph{equivariant} if \eqref{ev:2} happens.
\end{definition}

Note that a non-equivariant boolean insertion step has weight $1$ and changes the element, whereas an equivariant boolean insertion step picks up a nontrivial weight but does not modify the boolean elements. In Section~\ref{subsec:coh}, we focus only on the non-equivariant insertions for $H^*(G/B;\mathbb{Z})$.

\begin{definition}\label{def:path-wt}
For a boolean insertion path $u^{(0)}\op{\beta_1}u^{(1)}\op{\beta_2}\cdots\op{\beta_n}u^{(n)}$, its \emph{weight} is
\[\wt(u^{(0)}\op{\beta_1}u^{(1)}\op{\beta_2}\cdots\op{\beta_n}u^{(n)})\coloneqq\prod_{j=1}^{n}\wt(u^{(j-1)}\op{\beta_j}u^{(j)}).\]
For convenience, we also write the boolean insertion path above as $u^{(0)}\op{B}u^{(n)}$ where $B=\{\beta_1, \cdots, \beta_n\}$. Note that we need to always fix an ordering $B=(\beta_1, \cdots, \beta_n)$ first, before summing over boolean insertion paths $u^{(0)}\op{B}u^{(n)}$ for a fixed set $B$.
\end{definition}
\begin{remark}
The ordering of $B$ in Definition~\ref{def:path-wt} is arbitrary. Choosing an ordering strategically may be useful (see Section~\ref{sec:algo}).
\end{remark}

\begin{ex}\label{ex:op}
The directed Dynkin diagram of type $E_7$ is shown in Figure~\ref{fig:boolean-diagram-example}. Let $u=s_{_3}s_{_5}s_{_4}s_{_7}$ and $\di{u}$ be its boolean diagram indicated by the solid vertices and directed edges shown in Figure~\ref{fig:op-start}.
\begin{figure}[h!]
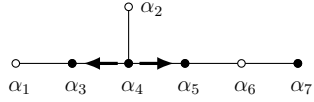

\centering
\begin{dynkinDiagram}[label,label macro/.code={\alpha_{\drlap{#1}}},edge length=.75cm,mark=o]E7
\foreach\r in {3,4,5,7} {\dynkinRootMark{*}\r}
\draw[very thick, black,-latex]
            (root 4) to (root 3);
\draw[very thick, black,-latex]
            (root 4) to (root 5);
\end{dynkinDiagram}
\caption{The boolean diagram $B(u)$ for the boolean element $u=s_3s_5s_4s_7$.}
\label{fig:op-start}
\end{figure}
Then there are $5$ boolean elements $v\in W$ satisfying $u\op{\alpha_4}v$, which correspond to all the boolean terms $\xi_v$ appearing in the expansion of $\xi_u \cdot \xi_{s_{\alpha_4}}$. One of them is $u$ itself with the equivariant step and $\wt(u\op{\alpha_4}u)=t_3+t_4+t_5$. The boolean diagrams of the other $4$ are shown in Figure~\ref{fig:op-end1}.
    
    \begin{figure}[h!]
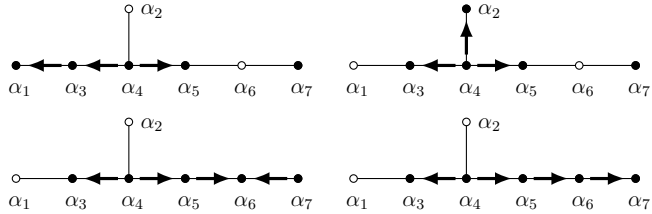

    \centering
    \begin{tabular}{cc}
    \begin{dynkinDiagram}[label,label macro/.code={\alpha_{\drlap{#1}}},edge length=.75cm,mark=o]E7
        \foreach\r in {1,3,4,5,7} {\dynkinRootMark{*}\r}
        \draw[very thick, black,-latex]
            (root 4) to (root 3);
        \draw[very thick, black,-latex]
            (root 4) to (root 5);
        \draw[very thick, black,-latex]
            (root 3) to (root 1);
    \end{dynkinDiagram}
    &\begin{dynkinDiagram}[label,label macro/.code={\alpha_{\drlap{#1}}},edge length=.75cm,mark=o]E7
        \foreach\r in {2,3,4,5,7} {\dynkinRootMark{*}\r}
        \draw[very thick, black,-latex]
            (root 4) to (root 3);
        \draw[very thick, black,-latex]
            (root 4) to (root 5);
        \draw[very thick, black,-latex]
            (root 4) to (root 2);
    \end{dynkinDiagram}\\
    \begin{dynkinDiagram}[label,label macro/.code={\alpha_{\drlap{#1}}},edge length=.75cm,mark=o]E7
        \foreach\r in {3,4,5,6,7} {\dynkinRootMark{*}\r}
        \draw[very thick, black,-latex]
            (root 4) to (root 3);
        \draw[very thick, black,-latex]
            (root 4) to (root 5);
        \draw[very thick, black,-latex]
            (root 5) to (root 6);
        \draw[very thick, black,-latex]
            (root 7) to (root 6);
    \end{dynkinDiagram}
    &\begin{dynkinDiagram}[label,label macro/.code={\alpha_{\drlap{#1}}},edge length=.75cm,mark=o]E7
        \foreach\r in {3,4,5,6,7} {\dynkinRootMark{*}\r}
        \draw[very thick, black,-latex]
            (root 4) to (root 3);
        \draw[very thick, black,-latex]
            (root 4) to (root 5);
        \draw[very thick, black,-latex]
            (root 5) to (root 6);
        \draw[very thick, black,-latex]
            (root 6) to (root 7);
    \end{dynkinDiagram}\\
    \end{tabular}
    \caption{The boolean diagrams $B(v)$ for all the boolean elements $v$ satisfying $s_3s_5s_4s_7=u\op{\alpha_4}v$ with a non-equivariant insertion step.}
    \label{fig:op-end1}
    \end{figure}
    Figure~\ref{fig:op-end1} gives an example where $\alpha\in\supp(u)$. Now choose $\alpha_6\notin\supp(u)$ and consider $u\op{\alpha_6}v$. Here, only non-equivariant insertion steps are possible. The diagrams of all of the boolean elements $v\in W$ satisfying $u\op{\alpha_6}v$ are shown in Figure~\ref{fig:op-end2}, which correspond to all the terms $\xi_v$ appearing in the expansion of $\xi_u \cdot \xi_{s_{\alpha_6}}$.
    \begin{figure}[h!]
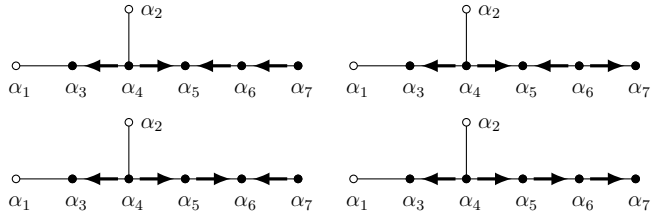

    \centering
    \begin{tabular}{cc}
    \begin{dynkinDiagram}[label,label macro/.code={\alpha_{\drlap{#1}}},edge length=.75cm,mark=o]E7
        \foreach\r in {3,4,5,6,7} {\dynkinRootMark{*}\r}
        \draw[very thick, black,-latex]
            (root 4) to (root 3);
        \draw[very thick, black,-latex]
            (root 4) to (root 5);
        \draw[very thick, black,-latex]
            (root 6) to (root 5);
        \draw[very thick, black,-latex]
            (root 7) to (root 6);
    \end{dynkinDiagram}
    &\begin{dynkinDiagram}[label,label macro/.code={\alpha_{\drlap{#1}}},edge length=.75cm,mark=o]E7
        \foreach\r in {3,4,5,6,7} {\dynkinRootMark{*}\r}
        \draw[very thick, black,-latex]
            (root 4) to (root 3);
        \draw[very thick, black,-latex]
            (root 4) to (root 5);
        \draw[very thick, black,-latex]
            (root 6) to (root 5);
        \draw[very thick, black,-latex]
            (root 6) to (root 7);
    \end{dynkinDiagram}\\
    \begin{dynkinDiagram}[label,label macro/.code={\alpha_{\drlap{#1}}},edge length=.75cm,mark=o]E7
        \foreach\r in {3,4,5,6,7} {\dynkinRootMark{*}\r}
        \draw[very thick, black,-latex]
            (root 4) to (root 3);
        \draw[very thick, black,-latex]
            (root 4) to (root 5);
        \draw[very thick, black,-latex]
            (root 5) to (root 6);
        \draw[very thick, black,-latex]
            (root 7) to (root 6);
    \end{dynkinDiagram}
    &\begin{dynkinDiagram}[label,label macro/.code={\alpha_{\drlap{#1}}},edge length=.75cm,mark=o]E7
        \foreach\r in {3,4,5,6,7} {\dynkinRootMark{*}\r}
        \draw[very thick, black,-latex]
            (root 4) to (root 3);
        \draw[very thick, black,-latex]
            (root 4) to (root 5);
        \draw[very thick, black,-latex]
            (root 5) to (root 6);
        \draw[very thick, black,-latex]
            (root 6) to (root 7);
    \end{dynkinDiagram}\\
    \end{tabular}
    \caption{The boolean diagrams $B(v)$ for $s_3s_5s_4s_7=u\op{\alpha_6}v$.}
    \label{fig:op-end2}
    \end{figure}
\end{ex}

\begin{ex}\label{ex:wt}
Consider a Dynkin type that is not simply-laced. Let $u=s_2s_3s_4$ in $W(C_4)$ shown in Figure~\ref{fig:wt}. Since there are $2$ edges from $\alpha_4$ to $\alpha_3$ of the directed Dynkin diagram by Definition~\ref{def:dynkin}, there are $2$ directed paths from $\alpha_4$ to $\alpha_3$ in the directed Dynkin diagram which are compatible with the direction of $B(u)$. Similarly, there are $2$ directed paths from $\alpha_4$ to $\alpha_2$. It follows that $\wt(u\op{\alpha_4}u)=2t_2+2t_3+t_4$.
    \begin{figure}[h!]
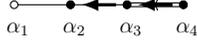

    \centering
    \begin{dynkinDiagram}[label,label macro/.code={\alpha_{\drlap{#1}}},edge length=.75cm,mark=o]C4
        \foreach\r in {2,3,4} {\dynkinRootMark{*}\r}
        \draw[very thick, black,-latex]
            (root 4) to (root 3);
        \draw[very thick, black,-latex]
            (root 3) to (root 2);
    \end{dynkinDiagram}
    \caption{The boolean diagrams $B(u)$ for $u=s_2s_3s_4$}
    \label{fig:wt}
    \end{figure}
\end{ex}
\begin{definition}\label{def:mul}
    For $u\op{\alpha}v$, define its \emph{multiplicity}, denoted by $\mul(u\op{\alpha}v)$, as follows:
    \begin{enumerate}
        \item If $u\op{\alpha}v$ is equivariant as in Definition~\ref{def:op}, $\mul(u\op{\alpha}v)\coloneqq1$.
        \item If $u\op{\alpha}v$ is non-equivariant as in Definition~\ref{def:op}, let $\gamma$ be the unique vertex of $B(v)\setminus B(u)$. Then $\mul(u\op{\alpha}v)$ is the number of directed paths from $\alpha$ to $\gamma$ in the directed Dynkin diagram which are compatible with the direction of $B(v)$.
    \end{enumerate}
\end{definition}
Note that when event~\eqref{ev:2} or \eqref{ev:3} in Definition~\ref{def:op} occurs, $\mul(u\op{\alpha}v)=1$.
\begin{definition}\label{def:path-mul}
    For a boolean insertion path $u^{(0)}\op{\beta_1}u^{(1)}\op{\beta_2}\cdots\op{\beta_n}u^{(n)}$,  its \emph{multiplicity} is the product of the multiplicities of all its steps.
\end{definition}
\begin{ex}\label{ex:mul}
    Consider $u=s_2s_3s_4$ in $W(C_4)$, where $B(u)$ is shown in Figure~\ref{fig:wt}. The boolean insertion $u\op{\alpha_4}v$ gives $v=s_1s_2s_3s_4$, where $B(v)$ is shown in Figure~\ref{fig:mul}. Now $\mul(u\op{\alpha_4}v)=2$ since there are $2$ directed paths from $\alpha_4$ to $\alpha_1$.
    \begin{figure}[h!]
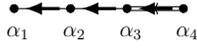

    \centering
    \begin{dynkinDiagram}[label,label macro/.code={\alpha_{\drlap{#1}}},edge length=.75cm,mark=o]C4
        \foreach\r in {1,2,3,4} {\dynkinRootMark{*}\r}
        \draw[very thick, black,-latex]
            (root 4) to (root 3);
        \draw[very thick, black,-latex]
            (root 3) to (root 2);
        \draw[very thick, black,-latex]
            (root 2) to (root 1);
    \end{dynkinDiagram}
    \caption{The boolean diagram $B(v)$ for $v=s_1s_2s_3s_4$.}
    \label{fig:mul}
    \end{figure}
\end{ex}
The following technical lemma is the basis of our calculations, which encodes Theorem~\ref{thm:chevalley} restricted to boolean elements.

\begin{lemma}\label{lem:monk}
    For $\alpha\in\Delta$ and a boolean element $v\in W$,
    \[\xi_v\cdot\xi_{s_\alpha}=\sum_{v\op{\alpha}w}\mul(v\op{\alpha}w)\wt(v\op{\alpha}w)\xi_w \mod \nb.\]
\end{lemma}
We give a definition that helps the proof of Lemma~\ref{lem:monk}.
\begin{definition}\label{def:dom}
    For a boolean element $w\in W$ and a vertex $\alpha\in B(w)$, the \emph{accessible subgraph} from $\alpha$ is the subgraph of the directed Dynkin diagram induced by vertices $\gamma\in B(w)$ which can be reached by a directed path of $B(w)$ from $\alpha$. Denote the accessible subgraph from $\alpha$ by $B(w,\alpha)$.
\end{definition}
For example, for $w=s_3s_2s_4s_5s_7$ in type $E_7$ shown In Figure~\ref{fig:boolean-diagram-example}, $B(w,\alpha_4)$ is the subgraph of the directed Dynkin diagram induced by $\alpha_2$, $\alpha_3$ and $\alpha_4$.
\begin{proof}[Proof of Lemma~\ref{lem:monk}]
Denote the directed Dynkin diagram by $D$ throughout the proof.

\
    
\noindent\textbf{Case (1)}: $\alpha\not\in\supp(v)$. It suffices to show that
\[\xi_v\cdot\xi_{s_{\alpha}}=\sum_{\supp(w)=\supp(v)\cup\{\alpha\}}\xi_w,\]
summing over boolean elements $w\in W$. In fact, if $\alpha_i\neq\alpha\in\Delta$, $s_i(\omega_{\alpha})=\omega_{\alpha}$ as $\langle\omega_{\alpha},\alpha_i\rangle=0$. This means $\omega_{\alpha}=v(\omega_{\alpha})$ if $\alpha\notin\supp(v)$, so equivariant term in Theorem~\ref{thm:chevalley} vanishes.

Now for any boolean element $w$ covering $v$ in the Bruhat order, there is a unique simple root $\gamma\in\supp(w)\setminus\supp(v)$. Then $\beta=u(\gamma)$ for some boolean element $u$ whose support is contained in $\supp(v)$. Since $\alpha\not\in\supp(v)$, $\alpha\not\in\supp(u)$. By Proposition~\ref{prop:game}, the coefficient of $\alpha$ in the expansion of $u(\gamma)$ is the same as $\gamma$. Thus $\langle\omega_{\alpha},\beta^{\vee}\rangle=0$ if $\gamma\neq\alpha$. If $\gamma=\alpha$, \[\langle\omega_{\alpha},\beta^{\vee}\rangle=\frac{2\langle\omega_{\alpha},\beta\rangle}{\langle\beta,\beta\rangle}=\frac{2\langle\omega_{\alpha},\gamma\rangle}{\langle\gamma,\gamma\rangle}=\frac{2\langle\omega_{\alpha},\alpha\rangle}{\langle\alpha,\alpha\rangle}=1.\]

Now assume that $\alpha\in\supp(v)$. Note that $v\op{\alpha}w$ can be either equivariant or non-equivariant as in Definition~\ref{def:op}, corresponding to different terms in Theorem~\ref{thm:chevalley}. We address these two situations as Case (2) and Case (3) respectively.

\

\noindent\textbf{Case (2)}: $\alpha\in\supp(v)$ and $\ell(v)=\ell(w)$. If $v\neq w$, then $d_{v,s_\alpha}^w = 0$ in both Theorem~\ref{thm:chevalley} and Lemma~\ref{lem:monk} which we are proving. Now assume that $v=w$, i.e. we have an equivariant boolean insertion $v\op{\alpha}w$. This means $\mul(v\op{\alpha}w)=1$. Referring to the equivariant term in Theorem~\ref{thm:chevalley}, we need to show that 
        \[\wt(v\op{\alpha}v)=\omega_{\alpha}-v(\omega_{\alpha})
        .\]
        Thus it suffices to show that for any simple root $\beta\in\Delta$,
        \begin{align}\label{eq:innerprod}
            \langle\wt(v\op{\alpha}v),\beta\rangle=\langle\omega_{\alpha}-v(\omega_{\alpha}),\beta\rangle.
        \end{align}
        The left hand side of \eqref{eq:innerprod} is
        \begin{equation}\label{eq:LHS'}
            \begin{aligned}
            \sum_{L}\langle t_{\gamma},\beta\rangle
            &=\sum_{L}\langle \omega_{\gamma}-s_{\gamma}(\omega_{\gamma}),\beta\rangle
            =\sum_L\langle\omega_{\gamma},\beta\rangle-\sum_L\langle s_{\gamma}(\omega_{\gamma}),\beta\rangle\\
            &=\sum_L\langle\omega_{\gamma},\beta\rangle-\sum_L\langle\omega_{\gamma},s_{\gamma}(\beta)\rangle
            =\sum_{L}\langle\omega_{\gamma},\beta-s_{\gamma}(\beta)\rangle\\
            &=\sum_{L}\left\langle\omega_{\gamma},\frac{2\langle\gamma,\beta\rangle}{\langle\gamma,\gamma\rangle}\gamma\right\rangle
            =\sum_L\langle\gamma,\beta\rangle.
            \end{aligned}
        \end{equation}
        summing over directed paths $L$ as in Definition~\ref{def:op}(\ref{ev:2}). The right hand side of \eqref{eq:innerprod} is
        \begin{equation}\label{eq:RHS}
            \begin{aligned}
                \langle\omega_{\alpha},\beta\rangle-\langle\omega_{\alpha},v^{-1}(\beta)\rangle.
            \end{aligned}
        \end{equation}
        
        We can calculate the coefficient of $\alpha$ in $v^{-1}(\beta)$ in \eqref{eq:RHS} by Proposition~\ref{prop:game}. Note that each $\gamma$ that shows up in \eqref{eq:LHS'} lies in $B(v,\alpha)$.
        Hence both \eqref{eq:LHS'} and \eqref{eq:RHS} are equal to $0$ if $\beta$ does not lie in the neighbourhood of $B(v,\alpha)$ in $D$ (``the neighbourhood of $B(v,\alpha)$" means all the vertices in $B(v,\alpha)$ or connected to some vertex in $B(v,\alpha)$ by some edge of $D$). Now assume that $\beta$ lies in the neighbourhood of $B(v,\alpha)$.\par
        \vspace{5pt}
        \noindent\textbf{Subcase (2.1)}: $\beta\not\in B(v,\alpha)$. There is a unique vertex $\gamma_0\in B(v,\alpha)$ adjacent to $\beta$. Suppose that there are $a$ directed paths from $\alpha$ to $\gamma_0$, $b$ directed paths from $\gamma_0$ to $\alpha$, $c$ edges from $\gamma_0$ to $\beta$, and $d$ edges from $\beta$ to $\gamma_0$ in $D$. Applying Proposition~\ref{prop:game} to \eqref{eq:LHS'}, the left hand side of \eqref{eq:innerprod} equals
        \[a\langle\gamma_0,\beta\rangle=a\cdot\left(-\frac{c}{2}\langle\beta,\beta\rangle\right)=-\frac{ac}{2}\langle\beta,\beta\rangle.  \] 
        Expression \eqref{eq:RHS} indicates that the right hand side of \eqref{eq:innerprod} equals
        \[-\langle\omega_{\alpha},v^{-1}(\beta)\rangle
                =-\langle\omega_{\alpha},bd\cdot\alpha\rangle
                =-\frac{bd}{2}\langle\alpha,\alpha\rangle=-\frac{bd}{2}\cdot\frac{ac}{bd}\langle\beta,\beta\rangle
                =-\frac{ac}{2}\langle\beta,\beta\rangle.\]

\
        
            \noindent\textbf{Subcase (2.2)}: $\beta\in B(v,\alpha)$ and $\beta\neq\alpha$. We have the following two scenarios. 
            
                If the directed path from $\alpha$ to $\beta$ in $B(w)$ cannot be extended, there is a unique vertex $\gamma_0\in B(v,\alpha)$ adjacent to $\beta$. Suppose that there are $a$ directed paths from $\alpha$ to $\gamma_0$, $b$ directed paths from $\gamma_0$ to $\alpha$, $c$ edges from $\gamma_0$ to $\beta$, and $d$ edges from $\beta$ to $\gamma_0$ in $D$.            \eqref{eq:LHS'} equals
                \[a\langle\gamma_0,\beta\rangle+ac\langle\beta,\beta\rangle=-\frac{ac}{2}\langle\beta,\beta\rangle+ac\langle\beta,\beta\rangle=\frac{ac}{2}\langle\beta,\beta\rangle.\]
                Apply Proposition~\ref{prop:game} to \eqref{eq:RHS}, then the coefficient of $\alpha$ in $v^{-1}(\beta)$ must equal to the number of directed paths in $D$ from $\beta$ to $\alpha$, and thus the right hand side of \eqref{eq:innerprod} becomes
                \[-\langle\omega_{\alpha},v^{-1}(\beta)\rangle
                        =-\langle\omega_{\alpha},-bd\cdot\alpha\rangle
                        =\frac{bd}{2}\langle\alpha,\alpha\rangle=\frac{bd}{2}\cdot\frac{ac}{bd}\langle\beta,\beta\rangle
                        =\frac{ac}{2}\langle\beta,\beta\rangle.\]

                If the directed path from $\alpha$ to $\beta$ in $B(w)$ can be extended, suppose that the out-degree of
                $\beta$ in $B(w)$ is $e$ corresponding to vertices $\gamma_j$ for $j=1,\cdots,e$. In $D$, suppose there are $a_j$ edges from $\beta$ to $\gamma_j$ and $b_j$ edges from $\gamma_j$ to $\beta$. Note that there is a unique vertex $\gamma_0\in B(v,\alpha)$ adjacent to $\beta$ pointing to $\beta$ in $B(w)$. Suppose that there are $a$ directed paths from $\alpha$ to $\gamma_0$, $b$ directed paths from $\gamma_0$ to $\alpha$, $c$ edges from $\gamma_0$ to $\beta$, and $d$ edges from $\beta$ to $\gamma_0$ in $D$.
                
                By \eqref{eq:LHS'}, the left hand side of \eqref{eq:innerprod} is
                \begin{equation*}
                    \begin{aligned}
                    \left(a\langle\beta,\gamma_0\rangle+ac\langle\beta,\beta\rangle+\sum_{j=1}^{e}aca_{j}\langle\beta,\gamma_j\rangle\right)
                    &=\left(-\frac{ac}{2}\langle\beta,\beta\rangle+ac\langle\beta,\beta\rangle-\sum_{j=1}^{e}\frac{aca_jb_j}{2}\langle\beta,\beta\rangle\right)\\
                    &=\frac{ac}{2}\left(1-\sum_{j=1}^{e}a_jb_j\right)\langle\beta,\beta\rangle.
                    \end{aligned}
                \end{equation*}
                Applying Proposition~\ref{prop:game} to \eqref{eq:RHS}, the right hand side of \eqref{eq:innerprod} becomes
                \begin{equation*}
                    \begin{aligned}
                        -\langle\omega_{\alpha},v^{-1}(\beta)\rangle
                        &=-\left\langle\omega_{\alpha},bd\left(\sum_{j=1}^{e}a_jb_j-1\right)\alpha\right\rangle
                        =-\frac{bd}{2}\left(\sum_{j=1}^{e}a_jb_j-1\right)\langle\alpha,\alpha\rangle\\
                        &=-\frac{bd}{2}\left(\sum_{j=1}^{e}a_jb_j-1\right)\cdot\frac{ac}{bd}\langle\beta,\beta\rangle
                        =\frac{ac}{2}\left(1-\sum_{j=1}^{e}a_jb_j\right)\langle\beta,\beta\rangle.
                    \end{aligned}
                \end{equation*}

\
            
            \noindent\textbf{Subcase (2.3)}: $\beta=\alpha$. 
                If the out-degree of $\alpha$ in $B(w)$ is $0$, \eqref{eq:LHS'} simplifies to
                    $\langle\alpha,\alpha\rangle$ and \eqref{eq:RHS} simplifies to $                  \langle\omega_{\alpha},\alpha\rangle-\langle\omega_{\alpha},-\alpha\rangle=\langle\alpha,\alpha\rangle$ so they are equal. 
                Say the out-degree of $\alpha$ in $B(w)$ is $e>0$ corresponding to vertices $\gamma_j$ for $j=1,\cdots,e$. In $D$, there are $a_j$ edges from $\beta$ to $\gamma_j$ and $b_j$ edges from $\gamma_j$ to $\beta$. \eqref{eq:LHS'} becomes
                \begin{equation*}
                \begin{aligned}
                    \left(\langle\alpha,\alpha\rangle+\sum_{j=1}^{e}a_j\langle\gamma_j,\alpha\rangle\right)
                    =\left(\langle\alpha,\alpha\rangle-\sum_{j=1}^{e}\frac{a_jb_j}{2}\langle\alpha,\alpha\rangle\right)
                    =\left(1-\frac{1}{2}\sum_{j=1}^{e}a_jb_j\right)\langle\alpha,\alpha\rangle.
                \end{aligned}
                \end{equation*}
                Applying Proposition~\ref{prop:game} to \eqref{eq:RHS}, we have
                \begin{equation*}
                    \begin{aligned}
                        \langle\omega_{\alpha},\alpha\rangle-\left\langle\omega_{\alpha},\left(\sum_{j=1}^{e}a_jb_j-1\right)\alpha\right\rangle
                        &=\frac{1}{2}\langle\alpha,\alpha\rangle-\frac{1}{2}\left(\sum_{j=1}^{e}a_jb_j-1\right)\langle\alpha,\alpha\rangle\\
                        &=\left(1-\frac{1}{2}\sum_{j=1}^{e}a_jb_j\right)\langle\alpha,\alpha\rangle.
                    \end{aligned}
                \end{equation*}

\
                
        \noindent\textbf{Case (3)}: $\alpha\in\supp(v)$ and $v\lessdot w$ in the Bruhat order. Compare the coefficient $d_{v,s_\alpha}^w$ of the equivariant Schubert class $\xi_w$ indexed by a boolean element in Theorem~\ref{thm:chevalley} and Lemma~\ref{lem:monk} which we are proving. Theorem~\ref{thm:chevalley} indicates that $d_{v,s_\alpha}^w = \langle\omega_\alpha,\beta^\vee\rangle$ if $v\lessdot w$ in the Bruhat order and $w=vs_\beta$ for some positive root $\beta\in\Phi^+$, while Lemma~\ref{lem:monk} states that $$d_{v,s_\alpha}^w=\begin{cases}
            \mul(v\op{\alpha}w)\wt(v\op{\alpha}w) = \mul(v\op{\alpha}w), &\text{if we have a non-equivariant $v\op{\alpha}w$} \\
            0, &\text{if $v\lessdot w$ but $v\op{\alpha}w$ does not hold}.
        \end{cases}$$ Consequently, we need to establish the following two facts for boolean elements $v,w$:
        \begin{enumerate}
            \item [Fact 1.] Let $v,w$ be boolean elements such that $v\lessdot w$ in the Bruhat order but $w$ does not satisfy $v\op{\alpha}w$. Write $w=vs_{\beta}$ for some positive root $\beta\in\Phi^+$. Then $\langle\omega_{\alpha},\beta^{\vee}\rangle=0$.
            \item [Fact 2.] Let $v,w$ be boolean elements such that $v\op{\alpha}w$ and $w=vs_{\beta}$ for some positive root $\beta\in\Phi^+$. Then $\langle\omega_{\alpha},\beta^{\vee}\rangle=\mul(v\op{\alpha}w)$.
        \end{enumerate}      
        Let $\gamma$ be the unique vertex in $B(w)\setminus B(v)$.
        
        For Fact 1, there are no directed paths from $\alpha$ to $\gamma$ in $B(w)$ by Definition~\ref{def:op}. Either $\alpha$ and $\gamma$ lie in distinct connected components of $B(w)$ (i.e. there are no undirected paths in $B(w)$ connecting $\alpha$ and $\gamma$), or there exists one edge between $\alpha$ and $\gamma$ pointing towards $\alpha$ in $B(w)$. Both cases indicate that there exists a reduced word of $w$ where $s_{\gamma}$ appears after $s_{\alpha}$. Let this reduced word be $as_{\gamma}s_{\alpha_1}\cdots s_{\alpha_m}$ where $\alpha_j\neq\alpha$ for $j\in[m]$ and $a$ is a subword. Then $as_{\alpha_1}\cdots s_{\alpha_m}$ is a reduced word of $v$. Thus $w=vs_{\alpha_m}\cdots s_{\alpha_1}s_{\gamma}s_{\alpha_1}\cdots s_{\alpha_m}$, $s_{\beta}=s_{\alpha_m}\cdots s_{\alpha_1}s_{\gamma}s_{\alpha_1}\cdots s_{\alpha_m}$ and  $\beta=s_{\alpha_m}\cdots s_{\alpha_1}(\gamma)$. By Proposition~\ref{prop:game}, the coefficient of $\alpha$ in $\beta$ is $0$ since $\alpha_j\neq\alpha$ for $j\in[m]$. We have $\langle\omega_{\alpha},\beta^{\vee}\rangle=0$ and Fact 1 holds.
        
        For Fact 2, we have a directed path $L: \alpha=\alpha_0,\alpha_1,\cdots, \alpha_m=\gamma$ from $\alpha$ to $\gamma$ in $B(w)$. Then one reduced word of $w$ is of the form
        $as_{\gamma}s_{\alpha_{m-1}}\cdots s_{\alpha_1}s_{\alpha}b$ where both $a$ and $b$ are subwords. Consequently, $as_{\alpha_{m-1}}\cdots s_{\alpha_1}s_{\alpha}b$ is a reduced word of $v$. Therefore, \[w=vb^{-1}s_{\alpha}s_{\alpha_1}\cdots s_{\alpha_{m-1}}s_{\gamma}s_{\alpha_{m-1}}\cdots s_{\alpha_1}s_{\alpha}b\] and we deduce that $\beta=b^{-1}s_{\alpha}s_{\alpha_1}\cdots s_{\alpha_{m-1}}(\gamma)$. Apply Proposition~\ref{prop:game} and since all simple reflections that appear in $b$ do not contribute to the coefficient of $\alpha$ in $\beta$, we know that the coefficient of $\alpha$ in $\beta$ equals $t$, the number of directed paths in $D$ from $\gamma$ to $\alpha$. Let $r$ be the number of directed paths of $D$ from $\alpha$ to $\gamma$. Then we have
        \[\langle\omega_{\alpha},\beta^{\vee}\rangle=\frac{2\langle\omega_{\alpha},\beta\rangle}{\langle\beta,\beta\rangle}=\frac{2\langle\omega_{\alpha},\beta\rangle}{\langle\gamma,\gamma\rangle}=\frac{2\langle\omega_{\alpha},t\cdot\alpha\rangle}{\langle\gamma,\gamma\rangle}=\frac{t\langle\alpha,\alpha\rangle}{\langle\gamma,\gamma\rangle}=\frac{t\cdot r}{t}=r.\]
        By Definition~\ref{def:mul}, $r=\mul(v\op{\alpha}w)$ so Fact 2 holds.
\end{proof}

\subsection{Multiplying Schubert classes indexed by boolean elements}\label{subsec:mult}
Recall that once we fix an ordering $B=(\beta_1, \cdots, \beta_n)$ where $\beta_1,\cdots,\beta_n\in\Delta$, a boolean insertion path $u=u^{(0)}\op{\beta_1}u^{(1)}\op{\beta_2}\cdots\op{\beta_n}u^{(n)}=w$ can be written as $u\op{B}w$. The following is a direct corollary of Lemma~\ref{lem:monk}, which is obtained from applying Lemma~\ref{lem:monk} on $\beta_1,\ldots,\beta_n$ step by step. The ordering of $B$ is arbitrary because of the commutative multiplication $\xi_{s_{\beta_i}}\cdot\xi_{s_{\beta_j}} = \xi_{s_{\beta_j}}\cdot\xi_{s_{\beta_i}}$.
\begin{corollary}\label{cor:monk}
    For boolean element $u\in W$ and a set of simple roots $B\subseteq\Delta$, fix an ordering $B=(\beta_1, \cdots, \beta_n)$ of $B$, then
    \[\xi_u\prod_{\beta\in B}\xi_{s_{\beta}}=\sum_{u\op{B}w}\mul(u\op{B}w)\wt(u\op{B}w)\xi_{w} \mod \nb\]
    summing over all boolean insertion paths $u=u^{(0)}\op{\beta_1}u^{(1)}\op{\beta_2}\cdots\op{\beta_n}u^{(n)}=w$.
\end{corollary}
For convenience, for $f\in H_T^*(G/B;\mathbb{Z})$, we write $[\xi_w]f$ for the coefficient of $\xi_w$ in $f$ expanded in the basis of the equivariant Schubert classes.

The following result is the last technical lemma for Theorem~\ref{thm:main}.
\begin{lemma}\label{lem:sym}
    For boolean elements $u,v,w\in W$ satisfying $u\op{\supp(v)}w$ and $v\op{\supp(u)}w$,
    \[[\xi_w](\xi_u\cdot\xi_v)=[\xi_w]\Big(\xi_u\prod_{\beta\in\supp(v)}\xi_{s_{\beta}}\Big)=[\xi_w]\Big(\Big(\prod_{\alpha\in\supp(u)}\xi_{s_{\alpha}}\Big)\xi_v\Big)=[\xi_w]\Big(\Big(\prod_{\alpha\in\supp(u)}\xi_{s_{\alpha}}\Big)\Big(\prod_{\beta\in\supp(v)}\xi_{s_{\beta}}\Big)\Big).\]
\end{lemma}
\begin{proof}
    By Lemma~\ref{lem:monk}, specifically Case (1), we have that
    \[\prod_{\alpha\in\supp(u)}\xi_{s_{\alpha}}=\sum_{\supp(u^{\prime})=\supp(u)}\xi_{u^{\prime}},\qquad\prod_{\beta\in\supp(v)}\xi_{s_{\beta}}=\sum_{\supp(v^{\prime})=\supp(v)}\xi_{v^{\prime}}\] summing over boolean elements $u^{\prime}$ and $v^{\prime}$. Thus, 
    \begin{align}\label{eq:vexp}\xi_u\prod_{\beta\in\supp(v)}\xi_{s_{\beta}}=\sum_{\supp(v^{\prime})=\supp(v)}\xi_u\xi_{v^{\prime}}    
    \end{align} where $v^{\prime}$ is boolean,
    \begin{align}\label{eq:uexp}\Big(\prod_{\alpha\in\supp(u)}\xi_{s_{\alpha}}\Big)\xi_v=\sum_{\supp(u^{\prime})=\supp(u)}\xi_{u^{\prime}}\xi_v
    \end{align} where $u^{\prime}$ is boolean, 
    and
    \begin{align}\label{eq:exp}
    \Big(\prod_{\alpha\in\supp(u)}\xi_{s_{\alpha}}\Big)\Big(\prod_{\beta\in\supp(v)}\xi_{s_{\beta}}\Big)=\sum_{u^{\prime},v^{\prime}}\xi_{u^{\prime}}\xi_{v^{\prime}}
    \end{align}
    summing over boolean elements $u^{\prime},v^{\prime}\in W$ satisfying $\supp(u^{\prime})=\supp(u)$ and $\supp(v^{\prime})=\supp(v)$.
    
    We claim that for distinct pairs of boolean elements $(u^{\prime},v^{\prime})\neq(u^{\prime\prime},v^{\prime\prime})$ with $\supp(u^{\prime})=\supp(u^{\prime\prime})$ and $\supp(v^{\prime})=\supp(v^{\prime\prime})$, there does not exist a boolean element $w'$ such that $\xi_{w'}$ appears in both $\xi_{u^{\prime}}\xi_{v^{\prime}}$ and $\xi_{u^{\prime\prime}}\xi_{v^{\prime\prime}}$. Assume for the sake of contradiction that $\xi_{w^{\prime}}$ appears in both, and assume $u^{\prime}\neq u^{\prime\prime}$ without lose of generality. Then $u^{\prime}\leq w^{\prime}$ and $B(u^{\prime})\subseteq B(w^{\prime})$. Similarly, $B(u^{\prime\prime})\subseteq B(w^{\prime})$. But this indicates that $B(w^{\prime})$ contains different subgraphs induced by the same vertex set $\supp(u^\prime) = \supp(u^{\prime\prime})$, a contradiction.

    Now let a boolean element $w$ be as in the lemma statement.
    By $u\op{\supp(v)}w$ and Corollary~\ref{cor:monk}, $\xi_w$ appears in the expansion of some term on the right hand side of \eqref{eq:vexp}. Similarly, $\xi_w$ appears in the expansion of some term on the right hand side of \eqref{eq:uexp}. At the same time, the claim above indicates that the common terms of the right hand sides of \eqref{eq:vexp} and \eqref{eq:uexp} all come from $\xi_u\xi_v$. As a result, $\xi_w$ appears in the expansion of $\xi_u\xi_v$. Moreover, its coefficients in \eqref{eq:vexp}, \eqref{eq:uexp} and \eqref{eq:exp} are all equal to $[\xi_w](\xi_u\xi_v)$.
\end{proof}
\begin{remark}
Lemma~\ref{lem:sym} demonstrates a very unique property of boolean elements. Let $u\in W$ be any element and let $\mathbf{u}$ be a reduced word of $u$. We know that $\prod_{\alpha\in\mathbf{u}}\xi_{s_{\alpha}}$ contains $\xi_u$ and a lot of other terms. In general, we expect \[[\xi_w](\xi_u\cdot\xi_v)<[\xi_w]\Big(\Big(\prod_{\alpha\in\mathbf{u}}\xi_{s_{\alpha}}\Big)\Big(\prod_{\beta\in\mathbf{v}}\xi_{s_{\beta}}\Big)\Big).\]
However, Lemma~\ref{lem:sym} tells us that when $w$ is boolean, which implies that the relevant $u$ and $v$ are also boolean, we have an equality so that the structure constants are manageable. 
\end{remark}

We are now ready to prove Theorem~\ref{thm:main}.
\begin{proof}[Proof of Theorem~\ref{thm:main}]
    By Lemma~\ref{lem:sym}, we have 
    \[d_{uv}^w=[\xi_w](\xi_u\cdot\xi_v)=[\xi_w]\Big(\xi_u\prod_{\beta\in\supp(v)}\xi_{s_{\beta}}\Big).\]
    We are done by applying Corollary~\ref{cor:monk} to the right hand side.

    Note that $u$ and $v$ play a symmetric role in Lemma~\ref{lem:sym}. We can interchange $u$ and $v$ in the theorem statement.
\end{proof}

\subsection{Structure constants in the cohomology ring $H^*(G/B;\mathbb{Z})$}\label{subsec:coh}
In this section, we only need to consider non-equivariant boolean insertions with weights equal to $1$.
\begin{proof}[Proof of Corollary~\ref{cor:main}]
The cohomology version can be derived from the equivariant cohomology version by setting $t_{\alpha}=0$ in Theorem~\ref{thm:main} for each simple root $\alpha\in\Delta$. This is equivalent to requiring each boolean insertion to be non-equivariant with weight $1$.
\end{proof}

The following result is an interesting property of boolean insertions.
\begin{proposition}\label{prop:insert}
In the case where the directed Dynkin diagram is a path, fix an ordering $B=(\beta_1, \cdots, \beta_n)$ of a set of simple roots $B\subseteq\Delta$, then there exists at most one non-equivariant boolean insertion path $u\op{B}w$ for any boolean elements $u,w\in W$.
\end{proposition}
\begin{proof}
Assume for the sake of contradiction that there are two distinct non-equivariant boolean insertion paths : $u=u^{(0)}\op{\beta_1}u^{(1)}\op{\beta_2}\cdots\op{\beta_n}u^{(n)}=w$ and $u=v^{(0)}\op{\beta_1}v^{(1)}\op{\beta_2}\cdots\op{\beta_n}v^{(n)}=w$. Choose the smallest $i_0$ such that $u^{(i_0)}\neq v^{(i_0)}$. Then $u^{(i_0-1)}=v^{(i_0-1)}$.

\
    
    \noindent\textbf{Case (1)}: $\beta_{i_0}\not\in\supp(u^{(i_0-1)})$. We have $\supp(u^{(i_0)})=\supp(v^{(i_0)})$ by Definition~\ref{def:op} but $B(u^{(i_0)})\neq B(v^{(i_0)})$. Hence $B(u^{(n)})$ and $B(v^{(n)})$ have distinct subgraphs induced by the same vertex set $\supp(u^{(i_0)})=\supp(v^{(i_0)})$, indicating that $B(u^{(n)})\neq B(v^{(n)})$, which is a contradiction.

\
    
    \noindent\textbf{Case (2)}: $\beta_{i_0}\in\supp(u^{(i_0-1)})$. Let $\gamma$ be the new vertex added in the insertion step $u^{(i_0-1)}\op{\beta_{i_0}}u^{(i_0)}$ and $\alpha_1$ be the new vertex added in the insertion step $v^{(i_0-1)}\op{\beta_{i_0}}v^{(i_0)}$.

\
    
    \noindent\textbf{Subcase (2.1)}: $\gamma=\alpha_1$. Then $u^{(i_0)}\neq v^{(i_0)}$ indicates that both $B(u^{(i_0)})$ and $B(v^{(i_0)})$ have a directed path from $\beta_{i_0}$ to $\gamma$ but the directions of next edges are opposite to each other by Definition~\ref{def:op} (Figure~\ref{fig:add-same-pt}). Hence $B(u^{(n)})$ and $B(v^{(n)})$ have distinct subgraphs induced by the same vertex set $\supp(u^{(i_0)})=\supp(v^{(i_0)})$, indicating that $B(u^{(n)})\neq B(v^{(n)})$, which is a contradiction.
    \begin{figure}[H]
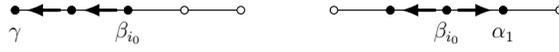

    \centering
    \begin{dynkinDiagram}[labels={,\gamma,\beta_{i_0},,},edge length=.75cm,mark=o]A5
        \foreach\r in {1,2,3} {\dynkinRootMark{*}\r}
        \draw[very thick, black,-latex]
            (root 3) to (root 2);
        \draw[very thick, black,-latex]
            (root 2) to (root 1);
    \end{dynkinDiagram}
    \qquad
    \begin{dynkinDiagram}[labels={,\gamma,\beta_{i_0},,},edge length=.75cm,mark=o]A5
        \foreach\r in {1,2,3} {\dynkinRootMark{*}\r}
        \draw[very thick, black,-latex]
            (root 3) to (root 2);
        \draw[very thick, black,-latex]
            (root 1) to (root 2);
    \end{dynkinDiagram}
    \caption{The boolean diagrams of $u^{(i_0)}$ and $v^{(i_0)}$.}
    \label{fig:add-same-pt}
    \end{figure}
    
    \
    
    \noindent\textbf{Subcase (2.2)}: $\gamma\neq\alpha_1$. In the directed Dynkin diagram, $B(u^{(i_0)})$ has more vertices than $B(v^{(i_0)})$ on the $\gamma$-side of $\beta_{i_0}$ (Figure~\ref{fig:add-dis-pt}). We show that $B(u^{(i)})$ has more vertices than $B(v^{(i)})$ on the $\gamma$-side of $\beta_{i_0}$ for $i\ge i_0$ by induction. If $\beta_{i+1}$ lies on the $\gamma$-side of $\beta_{i_0}$, then $u^{(i)}\op{\beta_{i+1}}u^{(i+1)}$ adds one vertex to the $\gamma$-side of $\beta_{i_0}$ since there exists one edge from $\beta_{i_0}$ to $\gamma$ in $B(u^{(i_0)})\subseteq B(u^{(i)})$. Note that $v^{(i)}\op{\beta_{i+1}}v^{(i+1)}$ adds at most one vertex to the $\gamma$-side of $\beta_{i_0}$. Hence $B(u^{(i+1)})$ has more vertices than $B(v^{(i+1)})$ on the $\gamma$-side of $\beta_{i_0}$. If $\beta_{i+1}$ lies on the $\alpha_1$-side of $\beta_{i_0}$, then $u^{(i)}\op{\beta_{i+1}}u^{(i+1)}$ does not decrease the number of vertices on the $\gamma$-side of $\beta_{i_0}$. Note that $v^{(i)}\op{\beta_{i+1}}v^{(i+1)}$ adds no vertices to the $\gamma$-side of $\beta_{i_0}$ since there exists one edge from $\beta_{i_0}$ to $\alpha_1$ in $B(v^{(i_0)})\subseteq B(v^{(i)})$. Hence $B(u^{(i+1)})$ has more vertices than $B(v^{(i+1)})$ on the $\gamma$-side of $\beta_{i_0}$. The induction step goes through. In particular, $B(u^{(n)})$ has more vertices than $B(v^{(n)})$ on the $\gamma$-side of $\beta_{i_0}$, which indicates that $u^{(n)}\neq v^{(n)}$. This is a contradiction.
\end{proof}
    \begin{figure}[H]
    \centering
    \begin{dynkinDiagram}[labels={\gamma,,\beta_{i_0},,},edge length=.75cm,mark=o]A5
        \foreach\r in {1,2,3} {\dynkinRootMark{*}\r}
        \draw[very thick, black,-latex]
            (root 3) to (root 2);
        \draw[very thick, black,-latex]
            (root 2) to (root 1);
    \end{dynkinDiagram}
    \qquad
    \begin{dynkinDiagram}[labels={,,\beta_{i_0},\alpha_1,},edge length=.75cm,mark=o]A5
        \foreach\r in {2,3,4} {\dynkinRootMark{*}\r}
        \draw[very thick, black,-latex]
            (root 3) to (root 2);
        \draw[very thick, black,-latex]
            (root 3) to (root 4);
    \end{dynkinDiagram}
    \caption{The boolean diagrams of $u^{(i_0)}$ and $v^{(i_0)}$.}
    \label{fig:add-dis-pt}
    \end{figure}
In type $A$, all the multiplicities as in Definition~\ref{def:mul} are $1$. Combining Corollary~\ref{cor:main} and Proposition~\ref{prop:insert}, we arrive at the following result.
\begin{corollary}\label{cor:0-1}
For boolean elements $u,v,w$ in the Weyl group of type $A$, $c_{uv}^w=1$ if there exist non-equivariant boolean insertion paths $u\op{\supp(v)}w$ and $v\op{\supp(u)}w$; $c_{uv}^{w}=0$ otherwise.
\end{corollary}

\section{Fast algorithms for computation}\label{sec:algo}
In this section, we work in type $A_n$. We provide Algorithm~\ref{algo} that determines whether there exists a non-equivariant boolean insertion path $u\op{\supp(v)}w$ for boolean elements $u,v,w\in W$. This algorithm works by finding a good ordering of $S(v)$ for the insertion paths. The correctness and the time complexity of the algorithm is provided in Theorem~\ref{thm:work}. By Corollary~\ref{cor:main}, we can calculate the structure constants $c_{uv}^w$ for boolean permutations in the symmetric group in $O(n^2)$ time as well. 

\begin{algorithm}[h!]
\begin{algorithmic}[1]
\caption{Construction of a boolean insertion path}\label{algo}
\REQUIRE Boolean elements $u,v,w\in W$.
\ENSURE A boolean insertion path $u\op{\supp(v)}w$ if it exists.
    \STATE Initialize $B=B(u)$, $P$ to be an empty list and $S=\supp(v)$. 
    \STATE Check whether $B(u)\subseteq B(w)$. If not, return $\none$.
    \STATE For each $i\in S$ in increasing order, try the boolean insertion $B\op{i}$. If there is only one possible insertion step $B\op{i} B'$ satisfying $B'\subseteq B(w)$, remove $i$ from $S$, append this step to $P$ and replace $B$ by $B'$. If no such $B'\subseteq B(w)$ exists, return $\none$.
    \STATE Repeat Step 3 until no such insertions are available.
    \STATE Let the remaining vertices in $S$ and $B(w)\setminus B$ be $i_1<i_2<\cdots <i_m$ and $j_1<j_2<\cdots <j_m$ respectively. For $k=1,\ldots,m$ in increasing order, do $B\op{i_k}B'$ such that the newly added vertex in $B'$ is exactly $j_k$ and that $B'\subseteq B(w)$. Append this sequence of insertions to $P$ if they exist and return $\none$ if not.
    \STATE Return $P$.
\end{algorithmic}
\end{algorithm}

\begin{ex}\label{ex:algo}
    Let $u=s_4s_3s_8s_{11}s_{12}$, $v=s_2s_3s_7s_6s_8s_{12}$ and $ w=s_7s_6s_5s_4s_2s_3s_9s_8s_{11}s_{13}s_{12}$. In Algorithm~\ref{algo}, we begin with $B=B(u_0)$ where $u_0=u$ and $S=\supp(v)=\{2,3,6,7,8,12\}$. The boolean diagrams of $u$ and $w$ and all steps in Algorithm~\ref{algo} are shown in Figure~\ref{fig:algo}. In the end, we obtain a boolean insertion path $u\op{\supp(v)}w$. In fact, there is a boolean insertion path $v\op{\supp(u)}w$ as well. Thus, by Corollary~\ref{cor:0-1}, $c_{uv}^w=1$.
    \begin{figure}[H]
    \centering
    \subfigure[In each step, we should keep $B\subseteq B(w)$]{
    \begin{tabular}{cc}
         $B(w)$
         & \begin{dynkinDiagram}[labels={1,2,3,4,5,6,7,8,9,10,11,12,13},edge length=.75cm,mark=o]A{13}
        \foreach\r in {2,3,4,5,6,7,8,9,11,12,13} {\dynkinRootMark{*}\r}
        \draw[very thick, black,-latex]
            (root 3) to (root 2);
        \draw[very thick, black,-latex]
            (root 3) to (root 4);
        \draw[very thick, black,-latex]
            (root 4) to (root 5);
        \draw[very thick, black,-latex]
            (root 5) to (root 6);
        \draw[very thick, black,-latex]
            (root 6) to (root 7);
        \draw[very thick, black,-latex]
            (root 8) to (root 7);
        \draw[very thick, black,-latex]
            (root 8) to (root 9);
        \draw[very thick, black,-latex]
            (root 12) to (root 11);
        \draw[very thick, black,-latex]
            (root 12) to (root 13);
    \end{dynkinDiagram}
    \end{tabular}
    }
    \subfigure[Initialize, $S=\{2,3,6,7,8,12\}$]{
    \begin{tabular}{cc}
         $B(u_0)$
         & \begin{dynkinDiagram}[labels={1,2,3,4,5,6,7,8,9,10,11,12,13},edge length=.75cm,mark=o]A{13}
        \foreach\r in {3,4,8,11,12} {\dynkinRootMark{*}\r}
        \draw[very thick, black,-latex]
            (root 3) to (root 4);
        \draw[very thick, black,-latex]
            (root 12) to (root 11);
    \end{dynkinDiagram}
    \end{tabular}
    }
    \subfigure[$u_0\op{2}u_1$ in Step 3, $S=\{3,6,7,8,12\}$]{
    \begin{tabular}{cc}
         $B(u_1)$&  \begin{dynkinDiagram}[labels={1,2,3,4,5,6,7,8,9,10,11,12,13},edge length=.75cm,mark=o]A{13}
        \foreach\r in {2,3,4,8,11,12} {\dynkinRootMark{*}\r}
        \draw[very thick, black,-latex]
            (root 3) to (root 2);
        \draw[very thick, black,-latex]
            (root 3) to (root 4);
        \draw[very thick, black,-latex]
            (root 12) to (root 11);
    \end{dynkinDiagram}
    \end{tabular}
    }
    \subfigure[$u_1\op{6}u_2$ in Step 3, $S=\{3,7,8,12\}$]{
    \begin{tabular}{cc}
         $B(u_2)$&\begin{dynkinDiagram}[labels={1,2,3,4,5,6,7,8,9,10,11,12,13},edge length=.75cm,mark=o]A{13}
        \foreach\r in {2,3,4,6,8,11,12} {\dynkinRootMark{*}\r}
        \draw[very thick, black,-latex]
            (root 3) to (root 2);
        \draw[very thick, black,-latex]
            (root 3) to (root 4);
        \draw[very thick, black,-latex]
            (root 12) to (root 11);
    \end{dynkinDiagram}
    \end{tabular}
    }
    \subfigure[$u_2\op{7}u_3$ in Step 3, $S=\{3,8,12\}$]{
    \begin{tabular}{cc}
         $B(u_3)$&\begin{dynkinDiagram}[labels={1,2,3,4,5,6,7,8,9,10,11,12,13},edge length=.75cm,mark=o]A{13}
        \foreach\r in {2,3,4,6,7,8,11,12} {\dynkinRootMark{*}\r}
        \draw[very thick, black,-latex]
            (root 3) to (root 2);
        \draw[very thick, black,-latex]
            (root 3) to (root 4);
        \draw[very thick, black,-latex]
            (root 12) to (root 11);
        \draw[very thick, black,-latex]
            (root 6) to (root 7);
        \draw[very thick, black,-latex]
            (root 8) to (root 7);
    \end{dynkinDiagram}
    \end{tabular}
    }
    \subfigure[$u_3\op{8}u_4$ in Step 3, $S=\{3,12\}$]{
    \begin{tabular}{cc}
         $B(u_4)$&\begin{dynkinDiagram}[labels={1,2,3,4,5,6,7,8,9,10,11,12,13},edge length=.75cm,mark=o]A{13}
        \foreach\r in {2,3,4,6,7,8,9,11,12} {\dynkinRootMark{*}\r}
        \draw[very thick, black,-latex]
            (root 3) to (root 2);
        \draw[very thick, black,-latex]
            (root 3) to (root 4);
        \draw[very thick, black,-latex]
            (root 12) to (root 11);
        \draw[very thick, black,-latex]
            (root 6) to (root 7);
        \draw[very thick, black,-latex]
            (root 8) to (root 7);
        \draw[very thick, black,-latex]
            (root 8) to (root 9);
    \end{dynkinDiagram} 
    \end{tabular}
    }
    \subfigure[$u_4\op{3}u_5$ in Step 5, $S=\{12\}$]{
    \begin{tabular}{cc}
         $B(u_5)$&\begin{dynkinDiagram}[labels={1,2,3,4,5,6,7,8,9,10,11,12,13},edge length=.75cm,mark=o]A{13}
        \foreach\r in {2,3,4,5,6,7,8,9,11,12} {\dynkinRootMark{*}\r}
        \draw[very thick, black,-latex]
            (root 3) to (root 2);
        \draw[very thick, black,-latex]
            (root 3) to (root 4);
        \draw[very thick, black,-latex]
            (root 12) to (root 11);
        \draw[very thick, black,-latex]
            (root 6) to (root 7);
        \draw[very thick, black,-latex]
            (root 8) to (root 7);
        \draw[very thick, black,-latex]
            (root 8) to (root 9);
        \draw[very thick, black,-latex]
            (root 4) to (root 5);
        \draw[very thick, black,-latex]
            (root 5) to (root 6);
    \end{dynkinDiagram} 
    \end{tabular}
    }
    \subfigure[$u_5\op{12}u_6$ in Step 5, $S=\varnothing$]{
    \begin{tabular}{cc}
         $B(u_6)$&\begin{dynkinDiagram}[labels={1,2,3,4,5,6,7,8,9,10,11,12,13},edge length=.75cm,mark=o]A{13}
        \foreach\r in {2,3,4,5,6,7,8,9,11,12,13} {\dynkinRootMark{*}\r}
        \draw[very thick, black,-latex]
            (root 3) to (root 2);
        \draw[very thick, black,-latex]
            (root 3) to (root 4);
        \draw[very thick, black,-latex]
            (root 12) to (root 11);
        \draw[very thick, black,-latex]
            (root 6) to (root 7);
        \draw[very thick, black,-latex]
            (root 8) to (root 7);
        \draw[very thick, black,-latex]
            (root 8) to (root 9);
        \draw[very thick, black,-latex]
            (root 4) to (root 5);
        \draw[very thick, black,-latex]
            (root 5) to (root 6);
        \draw[very thick, black,-latex]
            (root 12) to (root 13);
    \end{dynkinDiagram} 
    \end{tabular}
    }
    \caption{An example for Algorithm~\ref{algo}, that results in a boolean insertion path $u=u_0\op{2}u_1\op{6}u_2\op{7}u_3\op{8}u_4\op{3}u_5\op{12}u_6=w$.}
    \label{fig:algo}
    \end{figure}
\end{ex}
Now we explain why Algorithm~\ref{algo} works.
\begin{theorem}\label{thm:work}
    Algorithm~\ref{algo} returns a boolean insertion path $u\op{\supp(v)}w$ if it exists and otherwise returns $\none$. The runtime of Algorithm~\ref{algo} is $O(n^2)$ in type $A_n$.
\end{theorem}
\begin{proof}
By Corollary~\ref{cor:monk}, we can choose to insert $\supp(v)$ in any order that we like. 
Throughout the process, the boolean diagram $B$ keeps track of the boolean permutation, and the end goal is $B(w)$. Therefore, after each step, we always need to make sure that $B$ is a subgraph of $B(w)$. We have also initialized $P$ to keep track of the boolean insertion path, and $S$ to keep track of the residue vertices in $\supp(v)$ which have not been inserted into $B$. Step 1 and 2 are necessary, with total runtime $O(n)$.

In Step 3, if the newly added vertex by some step $B\op{i}B^{\prime}$ is unique (see Figure~\ref{fig:algopf} for visualization), we must add it and make sure that the newly obtained boolean diagram is a subgraph of $B(w)$. Thus, this insertion step is unique if exists. The runtime of one insertion is $O(n)$ and the total runtime of Step 3 and 4 is $O(n^2).$

After Step 4, for each $i\in S$, there are $2$ potential new vertices that can be added to $B$ by some $B\op{i}B^{\prime}$ (see Figure~\ref{fig:algopf}).
     \begin{figure}[h!]
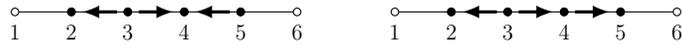

    \centering
    \begin{dynkinDiagram}[labels={1,2,3,4,5,6},edge length=.75cm,mark=o]A{6}
        \foreach\r in {2,3,4,5} {\dynkinRootMark{*}\r}
        \draw[very thick, black,-latex]
            (root 3) to (root 2);
        \draw[very thick, black,-latex]
            (root 3) to (root 4);
        \draw[very thick, black,-latex]
            (root 5) to (root 4);
    \end{dynkinDiagram}
    \qquad
    \begin{dynkinDiagram}[labels={1,2,3,4,5,6},edge length=.75cm,mark=o]A{6}
        \foreach\r in {2,3,4,5} {\dynkinRootMark{*}\r}
        \draw[very thick, black,-latex]
            (root 3) to (root 2);
        \draw[very thick, black,-latex]
            (root 3) to (root 4);
        \draw[very thick, black,-latex]
            (root 4) to (root 5);
    \end{dynkinDiagram}
    \caption{The unique new vertex can be added to the left boolean diagram $B$ by some $B\op{3}B'$ is $1$. All the possible new vertices can be added to the right boolean diagram by $\op{3}$ are $1$ and $6$.}
    \label{fig:algopf}
    \end{figure}
    Then we go to Step 5.
    Let $S=\{i_1<i_2<\cdots<i_m\}$ and let the vertices in $B(w)\setminus B$ be $j_1<j_2<\cdots<j_m$. We are now going to insert $S$ in this increasing order. Note that the new vertices added by $B\op{i_k}$ in the order $k=1,2,\cdots,m$ must lie on the directed Dynkin diagram in increasing order as well. Thus these new vertices must be $j_1,j_2,\cdots,j_m$ if there exists a boolean insertion path $u\op{\supp(v)}w$. The runtime of Step 5 is $O(n^2)$.

    Summing over the runtime of all the steps, we obtain $O(n^2)$.
\end{proof}
\begin{remark}
In arbitrary types, we can construct similar algorithms to find insertion paths and to compute structure constants $c_{uv}^w$ for boolean elements with the same time complexity $O(n^2)$, where $n=\mathrm{rank}(\Phi)\coloneqq\lvert\Delta\rvert$, the number of simple roots associated to the root system $\Phi$. Other directed Dynkin diagrams require analysis of edge cases on vertices with higher degrees, but the general idea stays the same.
\end{remark}
\section*{Acknowledgements}
We thank Prof. Anders Buch's equivariant Schubert calculator for calculations and we thank Weihong Xu, Rui Xiong and Alex Yong for pointing to us helpful references. Y.G is partially supported by NSFC Grant no. 12471309.

\bibliographystyle{plain}
\bibliography{ref}
\end{document}